\documentclass[oneside,10pt]{article}          
\usepackage[b5paper]{geometry}	    
\usepackage{amsfonts,amsmath,latexsym,amssymb} 
\usepackage{theorem}                
\usepackage{mathrsfs,upref}         
\usepackage{mathptmx}		    
\usepackage{mia}	            

\newtheorem{theorem}{Theorem}           
               
\newtheorem{corollary}{Corollary}

\theoremstyle{definition}

\newtheorem{example}{Example}

\newtheorem{remark}{Remark}
\newtheorem{prop}{Proposition}

\begin{document}

\title[Estimates for Tsallis relative operator entropy]{Estimates for Tsallis relative operator entropy}


\author{Hamid Reza Moradi, Shigeru Furuichi and Nicu\c sor Minculete}

\address{Young Researchers and Elite Club, Mashhad Branch, Islamic Azad University, Mashhad, Iran.\\
\email{hrmoradi@mshdiau.ac.ir}}

\address{Department of Information Science, College of Humanities and Sciences, Nihon University, 3-25-40, Sakurajyousui, Setagaya-ku, Tokyo, 156-8550, Japan.\\
\email{furuichi@chs.nihon-u.ac.jp}}

\address{Faculty of Mathematics and Computer Science, Transilvania University of Bra\c sov, Str. Iuliu Maniu, nr. 50, Bra\c sov, 500091, Romania.\\
\email{minculeten@yahoo.com}}

\CorrespondingAuthor{Shigeru Furuichi}


\date{04.03.2017}                               



\thanks{The authors would like to thank editor and referee for providing valuable comments to improve our manuscript. The author (S.F.) was partially supported by JSPS KAKENHI Grant Number 16K05257} 

\begin{abstract}
We give the tight bounds of Tsallis relative operator entropy by using Hermite-Hadamard's inequality. Some reverse inequalities related to Young inequalities are also given. In addition, operator inequalities for normalized positive linear map with Tsallis relative operator entropy are given.
\end{abstract}

\maketitle



\section{Introduction and Preliminaries}

The operator theory related to inequalities in Hilbert space is studied in many papers. Let $A, B$ be two operators in a Hilbert space $\mathscr{H}$. An operator $A$ is said to be strictly positive (denoted by $A>0$) if $A$ is positive and invertible. For two strictly positive operators $A,B$ and $p \in \left[ 0,1 \right]$, $p $-power mean $A{{\#}_{p }}B$ is defined by 
\[A{{\#}_{p}}B:={{A}^{\frac{1}{2}}}{{\left( {{A}^{-\frac{1}{2}}}B{{A}^{-\frac{1}{2}}} \right)}^{p }}{{A}^{\frac{1}{2}}},\] 
and we remark that $A{{\#}_{p }}B={{A}^{1-p }}{{B}^{p }}$ if $A$ commutes with $B$. 
We also use the symbol $A\natural_r B :={{A}^{\frac{1}{2}}}{{\left( {{A}^{-\frac{1}{2}}}B{{A}^{-\frac{1}{2}}} \right)}^{r }}{{A}^{\frac{1}{2}}}$ for $r \in \mathbb{R}$.
The weighted operator arithmetic mean is defined by
\[A{{\nabla }_{p}}B:=\left( 1-p \right)A+pB,\] 
for any $p\in \left[ 0,1 \right]$.

The relative operator entropy $S\left( A|B \right)$ in \cite{1} is defined by
\[S\left( A|B \right):={{A}^{\frac{1}{2}}}\left( \log {{A}^{-\frac{1}{2}}}B{{A}^{-\frac{1}{2}}} \right){{A}^{\frac{1}{2}}},\] 
where $A$ and $B$ are two invertible positive operators on a Hilbert space. As a parametric extension of the relative operator entropy, Yanagi, Kuriyama and Furuichi \cite{4} defined Tsallis relative operator entropy which is an operator version of Tsallis relative entropy in quantum system due to Abe \cite{abe}, also see \cite{2}:
\[{{T}_{p}}\left( A|B \right):=\frac{{{A}^{\frac{1}{2}}}{{\left( {{A}^{-\frac{1}{2}}}B{{A}^{-\frac{1}{2}}} \right)}^{p}}{{A}^{\frac{1}{2}}}-A}{p},\quad p\in \left( 0,1 \right].\]
Notice that ${{T}_{p }}\left( A|B \right)$ can be written by using the notation of $A{{\#}_{p}}B$ as follows: 
\begin{equation*}
{{T}_{p}}\left( A|B \right):=\frac{A{{\#}_{p}}B-A}{p},\quad p\in \left( 0,1 \right].
\end{equation*}

The relation between relative operator entropy $S\left( A|B \right)$ and Tsallis relative operator entropy ${{T}_{p}}\left( A|B \right)$ was considered in \cite{2,zou}, as follows:
\begin{equation}\label{39}
A-A{{B}^{-1}}A\le {{T}_{-p}}\left( A|B \right)\le S\left( A|B \right)\le {{T}_{p}}\left( A|B \right)\le B-A.
\end{equation}

Some deeper properties of Tsallis relative operator entropy were proved in \cite{7,8,9,zou}.
\vskip 0.3 true cm
The main result of the present paper is a set of bounds that are complementary to \eqref{39}. Some of our inequalities improve well-known ones. Among other inequalities, it is shown that if $A,B$ are invertible positive operators and $p\in \left( 0,1 \right]$, then
\[\begin{aligned}
   {{A}^{\frac{1}{2}}}{{\left( \frac{{{A}^{-\frac{1}{2}}}B{{A}^{-\frac{1}{2}}}+I}{2} \right)}^{p-1}}\left( {{A}^{-\frac{1}{2}}}B{{A}^{-\frac{1}{2}}}-I \right){{A}^{\frac{1}{2}}}&\le {{T}_{p}}\left( A|B \right) \\ 
 & \le \frac{1}{2}\left( A{{\#}_{p}}B-A{{\natural}_{p-1}}B+B-A \right),  
\end{aligned}\]
which is a considerable refinement of \eqref{39}, where $I$ is the identity operator.
We also  prove a reverse inequality involving Tsallis relative operator entropy ${{T}_{p}}\left( A|B \right)$.

\section{Refinements of the inequalities (\ref{39}) via Hermite-Hadamard inequality}
An important ingredient in our approach is the following:
 
Let $f$ be a convex function on $\left[ a,b \right]\subseteq \mathbb{R}$; the well-known Hermite-Hadamard's inequality  can be expressed as
\begin{equation} \label{original_HH_ineq}
f\left( \frac{a+b}{2} \right)\le \frac{1}{b-a}\int\limits_{a}^{b}{f\left( t \right)dt}\le \frac{f\left( a \right)+f\left( b \right)}{2}.
\end{equation} 
\begin{theorem}\label{9}
For any invertible positive operator $A$ and $B$ such that $A\le B$, and $p\in \left( 0,1 \right]$ we have
\begin{equation}\label{3}
\begin{aligned}
   {{A}^{\frac{1}{2}}}{{\left( \frac{{{A}^{-\frac{1}{2}}}B{{A}^{-\frac{1}{2}}}+I}{2} \right)}^{p-1}}\left( {{A}^{-\frac{1}{2}}}B{{A}^{-\frac{1}{2}}}-I \right){{A}^{\frac{1}{2}}}&\le {{T}_{p}}\left( A|B \right) \\ 
 & \le \frac{1}{2}\left( A{{\#}_{p}}B-A{{\natural}_{p-1}}B+B-A \right).  
\end{aligned}
\end{equation}
\end{theorem}
\begin{proof}
Consider the function $f\left( t \right)={{t}^{p-1}},\text{ }p\in \left( 0,1 \right]$. It is easy to check that $f\left( t \right)$ is convex on $\left[ 1, \infty  \right)$. Bearing in mind the fact
	\[\int\limits_{1}^{x}{{{t}^{p-1}}}dt=\frac{{{x}^{p}}-1}{p},\] 
and utilizing the left-hand side of Hermite-Hadamard inequality, one can see that  
\begin{equation}\label{1}
{{\left( \frac{x+1}{2} \right)}^{p-1}}\left( x-1 \right)\le \frac{{{x}^{p}}-1}{p},
\end{equation}
where $x\ge 1$ and $p\in \left( 0,1 \right]$. On the other hand, it follows from the right-hand side of Hermite-Hadamard inequality that
\begin{equation}\label{01}
\frac{{{x}^{p}}-1}{p}\le \left( \frac{{{x}^{p-1}}+1}{2} \right)\left( x-1 \right),
\end{equation}
for each $x\ge 1$ and $p\in \left( 0,1 \right]$. 

Replacing $x$ by ${{A}^{-\frac{1}{2}}}B{{A}^{-\frac{1}{2}}}$ in \eqref{1} and \eqref{01},
and multiplying ${{A}^{\frac{1}{2}}}$ on both sides, we get the desired result \eqref{3}.
\end{proof}
\begin{remark}\label{2.5}
Simple calculation gives for all $x\ge 1$ and $p\in \left( 0,1 \right]$,
\begin{equation} \label{ineq_remark23}
0\le 1-\frac{1}{x}\le {{\left( \frac{x+1}{2} \right)}^{p-1}}\left( x-1 \right)\le \frac{{{x}^{p}}-1}{p}\le \left( \frac{{{x}^{p-1}}+1}{2} \right)\left( x-1 \right)\le x-1,
\end{equation}
which means 
\[\begin{aligned}
 0\le  A-A{{B}^{-1}}A&\le {{A}^{\frac{1}{2}}}{{\left( \frac{{{A}^{-\frac{1}{2}}}B{{A}^{-\frac{1}{2}}}+I}{2} \right)}^{p-1}}\left( {{A}^{-\frac{1}{2}}}B{{A}^{-\frac{1}{2}}}-I \right){{A}^{\frac{1}{2}}} \\ 
 & \le {{T}_{p}}\left( A|B \right) \\ 
 & \le \frac{1}{2}\left( A{{\#}_{p}}B-A{{\natural}_{p-1}}B+B-A \right) \\ 
 & \le B-A, \\ 
\end{aligned}\]
for any invertible positive operators $A$ and $B$ such that $A\le B$, and $p\in \left( 0,1 \right]$. 
Therefore, our inequalities \eqref{3} improve the inequalities \eqref{39} for the case $A \leq B.$

\end{remark}
 \begin{prop} \label{proposition2_1}
  For $x\ge 1$ and $\frac{1}{2} \leq p \leq 1$,
\begin{equation}\label{ineq01_remark2_6}
  \frac{x-1}{\sqrt{x}} \leq \left( \frac{x+1}{2}\right)^{p-1}(x-1).
\end{equation}
 \end{prop}
\begin{proof}
  In order to prove \eqref{ineq01_remark2_6}, we set the function ${{f}_{p}}\left( x \right)\equiv {{\left( \frac{x+1}{2} \right)}^{p-1}}-\frac{1}{\sqrt{x}}$ where $x\geq 1$ and $\frac{1}{2} \leq p \leq 1$. Since $\frac{df_p(x)}{dp} = \left( \frac{x+1}{2}\right)^{p-1} \ln \left( \frac{x+1}{2}\right)$, $\frac{df_p(x)}{dp} \geq 0$ for $x \geq 1$. Thus, we have $f_p(x) \geq f_{1/2}(x) =\frac{\sqrt{2x} -\sqrt{x+1}}{\sqrt{x(x+1)}} \geq 0 $ for $x \geq 1$. Therefore, we have the inequality (\ref{ineq01_remark2_6}).
 \end{proof}

\begin{remark}
The first inequality \eqref{3} gives tight lower bound for the Tsallis relative operator entropy $T_p(A|B)$ more than the eighth inequality in {{\cite[Theorem 2.8 (i)]{pre}}}, due to Proposition \ref{proposition2_1}.
\end{remark}

\begin{corollary} \label{corollary2_1}
For any invertible positive operators $A$ and $B$ such that $A \geq B$, and $p \in (0,1]$, we have
\begin{equation*}
\begin{aligned}
A{{\#}_{p}}B-A{{\natural}_{p-1}}B & \le
\frac{1}{2}\left( A{{\#}_{p}}B-A{{\natural}_{p-1}}B+B-A \right)  \le {{T}_{p}}\left( A|B \right) \\ 
 & \le {{A}^{\frac{1}{2}}}{{\left( \frac{{{A}^{-\frac{1}{2}}}B{{A}^{-\frac{1}{2}}}+I}{2} \right)}^{p-1}}\left( {{A}^{-\frac{1}{2}}}B{{A}^{-\frac{1}{2}}}-I \right){{A}^{\frac{1}{2}}}\le  A{{\natural}_{p+1}}B-A{{\#}_{p}}B \le 0.  
\end{aligned}
\end{equation*}
\end{corollary}
\begin{proof}
Put $t = \frac{1}{x} \leq 1$ in the inequalities (\ref{ineq_remark23}).
\end{proof}

\begin{remark}
By numerical computations, we have  $\left( \frac{x^{p-1} +1}{2}\right)(x-1)- \frac{2(x-1)}{x+1} \simeq -0.83219$ for $p=\frac{2}{3}, x=0.01$, and we also have   $\left( \frac{x^{p-1} +1}{2}\right)(x-1)- \frac{2(x-1)}{x+1} \simeq 0.216868$ for $p=\frac{2}{3}, x=0.1$.
Thus we do not have ordering between $\left( \frac{x^{p-1} +1}{2}\right)(x-1)$ and $\frac{2(x-1)}{x+1}$ for $0<x\leq 1$ and $\frac{1}{2} \leq p \leq 1$ so that
there is no ordering between the second inequality in Corollary \ref{corollary2_1} and the sixth inequality in {{\cite[Theorem 2.8 (ii)]{pre}}}. 
\end{remark}




\begin{theorem}
For any invertible positive operators $A$ and $B$ such that $A\leq B$ and $p\in \left( 0,1 \right]$, we have the following inequalities,
\begin{equation} \label{ineq03_remark23}
{L_p}\left( {A,B} \right) + {K_p}\left( {A,B} \right) \le {T_p}\left( {A|B} \right) \le {R_p}\left( {A,B} \right) + {K_p}\left( {A,B} \right),
\end{equation}
and
\begin{equation} \label{ineq04_remark23}
{J_p}\left( {A,B} \right) - 2{R_p}\left( {A,B} \right) \le {T_p}\left( {A|B} \right) \le {J_p}\left( {A,B} \right) - 2{L_p}\left( {A,B} \right),
\end{equation}
where
\[\begin{array}{l}
{K_p}\left( {A,B} \right) \equiv {A^{1/2}}{\left( {\frac{{{A^{ - 1/2}}B{A^{ - 1/2}} + I}}{2}} \right)^{p - 1} \left( A^{-1/2}BA^{-1/2} -I\right)}{A^{1/2}},\\
{J_p}\left( {A,B} \right) \equiv \frac{1}{2}\left( {A{\# _p}B - A{\natural _{p - 1}}B + B - A} \right),\\
{L_p}\left( {A,B} \right) \equiv \frac{1}{{24}}\left( {p - 1} \right)\left( {p - 2} \right)\left( {A{\# _p}B - 3A{\natural _{p - 1}}B + 3A{\natural _{p - 2}}B - A{\natural _{p - 3}}B} \right),\\
{R_p}\left( {A,B} \right) \equiv \frac{1}{{24}}\left( {p - 1} \right)\left( {p - 2} \right)\left( {A{\natural _3}B - 3A{\natural _2}B + 3B - A} \right).
\end{array}\]
\end{theorem}

\begin{proof}
According to {{\cite[Theorem 1]{cn}}}, if $f:\left[ a,b \right]\to \mathbb{R}$ is a twice differentiable function that there exists real constants $m$ and $M$ so that $m\le f''\le M$, then 
\begin{equation} \label{ineq01_remark23}
m\frac{{{\left( b-a \right)}^{2}}}{24}\le \frac{1}{b-a}\int\limits_{a}^{b}{f\left( t \right)dt}-f\left( \frac{a+b}{2} \right)\le M\frac{{{\left( b-a \right)}^{2}}}{24}, 
\end{equation}
\begin{equation} \label{ineq02_remark23}
m\frac{{{\left( b-a \right)}^{2}}}{12}\le \frac{f\left( a \right)+f\left( b \right)}{2}-\frac{1}{b-a}\int\limits_{a}^{b}{f\left( t \right)dt}\le M\frac{{{\left( b-a \right)}^{2}}}{12}.
\end{equation}
Putting $f(t)=t^{p-1}$ with $p\in (0,1]$ and $a=1$, $b=x$ in the above inequalities, then we have the desired results by a
similar way to the proof of Theorem \ref{9}.
\end{proof}

\begin{remark}
The first inequality of (\ref{ineq03_remark23}) and the second inequality of (\ref{ineq04_remark23}) give  tighter bounds of  Tsallis relative entropy $T_p(A|B)$ than those in the inequalities (\ref{3}), because of the following reasons.
\begin{itemize}
\item[(i)] The first inequality of (\ref{ineq01_remark23}) gives tight lower bound more than the first inequality of Hermite-Hadamard's inequality (\ref{original_HH_ineq}).
\item[(ii)] The first inequality of (\ref{ineq02_remark23}) gives tight upper bound more than the second inequality of Hermite-Hadamard's inequality (\ref{original_HH_ineq}).
\end{itemize}
\end{remark}

\section{Some reverse inequalities via Young type inequalities}

The scalar Young's inequality says that if $a,b>0$, then
\[{{a}^{1-p}}{{b}^{p}}\le \left( 1-p \right)a+pb,\quad p\in \left[ 0,1 \right].\]
The following inequalities provide a refinement and a multiplicative reverse for Young's inequality with Kantorovich constant:
\begin{equation}\label{4}
{{K}^{r}}\left( h,2 \right){{a}^{1-p}}{{b}^{p}}\le \left( 1-p \right)a+pb\le {{K}^{R}}\left( h,2 \right){{a}^{1-p}}{{b}^{p}},
\end{equation}
where $a,b>0$, $p\in \left[ 0,1 \right]$, $r=\min \left\{ p,1-p \right\}$, $R=\max \left\{ p,1-p \right\}$, $K\left( h,2 \right)=\frac{{{\left( h+1 \right)}^{2}}}{4h}$ and $h=\frac{b}{a}$.  Notice that the first inequality in \eqref{4} was obtained by Zou {\it et al.} in {{\cite[Corollary 3]{21}}} while the second was obtained by Liao {\it et al.} {{\cite[Corollary 2.2]{20}}}.

In \cite{k1,k2}, Kittaneh and Manasrah obtained another refinement and reverse of Young's inequality:
\begin{equation}\label{25}
r{{\left( \sqrt{a}-\sqrt{b} \right)}^{2}}+{{a}^{1-p}}{{b}^{p}}\le \left( 1-p \right)a+pb\le R{{\left( \sqrt{a}-\sqrt{b} \right)}^{2}}+{{a}^{1-p}}{{b}^{p}},
\end{equation}
where $a,b>0$, $p\in \left[ 0,1 \right]$, $r=\min \left\{ p,1-p \right\}$, $R=\max \left\{ p,1-p \right\}$. Further refinements and generalizations of Young's inequality  have been obtained in \cite{al,mi}.

More interesting things happen when we apply these considerations to the operators. For instance, from the inequality \eqref{4} it follows that:
\begin{prop}\label{prop}
Let $A,B$ be two invertible positive operators such that $I<h'I\le {{A}^{-\frac{1}{2}}}B{{A}^{-\frac{1}{2}}}\le hI$ or $0<hI\le {{A}^{-\frac{1}{2}}}B{{A}^{-\frac{1}{2}}}\le h'I<I$, then 
\begin{equation}\label{23}
{{K}^{r}}\left( h',2 \right)A{{\#}_{p}}B\le A{{\nabla }_{p}}B\le {{K}^{R}}\left( h,2 \right)A{{\#}_{p}}B,
\end{equation}
where $p\in \left[ 0,1 \right]$, $r=\min \left\{ p,1-p \right\}$, $R=\max \left\{ p,1-p \right\}$.  
\end{prop}
\begin{proof}
The choice $a=1$ and $b=x$ reduces \eqref{4} to the inequality
\[{{K}^{r}}\left( x,2 \right){{x}^{p}}\le \left( 1-p \right)+px\le {{K}^{R}}\left( x,2 \right){{x}^{p}}.\]
Whence
\begin{equation}\label{21}
\underset{1<h'\le x\le h}{\mathop{\min }}\,{{K}^{r}}\left( x,2 \right){{T}^{p}}\le \left( 1-p \right)I+pT\le \underset{1<h'\le x\le h}{\mathop{\max }}\,{{K}^{R}}\left( x,2 \right){{T}^{p}},
\end{equation}
for any positive operator $T$ such that $I<h'I\le T\le hI$. On choosing $T={{A}^{-\frac{1}{2}}}B{{A}^{-\frac{1}{2}}}$ in \eqref{21} we get
\[\underset{1<h'\le x\le h}{\mathop{\min }}\,{{K}^{r}}\left( x,2 \right){{\left( {{A}^{-\frac{1}{2}}}B{{A}^{-\frac{1}{2}}} \right)}^{p}}\le \left( 1-p \right)I+p{{A}^{-\frac{1}{2}}}B{{A}^{-\frac{1}{2}}}\le \underset{1<h'\le x\le h}{\mathop{\max }}\,{{K}^{R}}\left( x,2 \right){{\left( {{A}^{-\frac{1}{2}}}B{{A}^{-\frac{1}{2}}} \right)}^{p}}.\]
Since $K\left( x,2 \right)$ is increasing for $x>1$ we can write
\begin{equation}\label{22}
{{K}^{r}}\left( h',2 \right){{\left( {{A}^{-\frac{1}{2}}}B{{A}^{-\frac{1}{2}}} \right)}^{p}}\le \left( 1-p \right)I+p{{A}^{-\frac{1}{2}}}B{{A}^{-\frac{1}{2}}}\le {{K}^{R}}\left( h,2 \right){{\left( {{A}^{-\frac{1}{2}}}B{{A}^{-\frac{1}{2}}} \right)}^{p}}.
\end{equation}
Multiplying both sides by ${{A}^{\frac{1}{2}}}$ to inequality \eqref{22}, we obtain the required inequality \eqref{23}.

Another case follows from the fact that $K\left( x,2 \right)$ is decreasing for $x<1$.   
\end{proof}
\vskip 0.3 true cm
Ando's inequality {{\cite[Theorem 3]{62}}} says that if $A,B$ are positive operators and $\Phi $ is a positive linear mapping, then
\begin{equation}\label{24}
\Phi \left( A{{\#}_{p}}B \right)\le \Phi \left( A \right){{\#}_{p}}\Phi \left( B \right),\quad p\in \left[ 0,1 \right].
\end{equation}
Concerning inequality \eqref{24}, we have the following corollary:
\begin{corollary}\label{7}
Let $A,B$ be two invertible positive operators such that $I<h'I\le {{A}^{-\frac{1}{2}}}B{{A}^{-\frac{1}{2}}}\le hI$ or $0<hI\le {{A}^{-\frac{1}{2}}}B{{A}^{-\frac{1}{2}}}\le h'I<I$. Let $\Phi $ be positive linear map on $\mathcal{B}\left( \mathscr{H} \right)$, then 
\begin{equation}\label{14}
\begin{aligned}
   \frac{{{K}^{r}}\left( h',2 \right)}{{{K}^{R}}\left( h,2 \right)}\Phi \left( A{{\#}_{p}}B \right)&\le \frac{1}{{{K}^{R}}\left( h,2 \right)}\Phi \left( A{{\nabla }_{p}}B \right) \\ 
 & \le \Phi \left( A \right){{\#}_{p}}\Phi \left( B \right) \\ 
 & \le \frac{1}{{{K}^{r}}\left( h',2 \right)}\Phi \left( A{{\nabla }_{p}}B \right) \\ 
 & \le \frac{{{K}^{R}}\left( h,2 \right)}{{{K}^{r}}\left( h',2 \right)}\Phi \left( A{{\#}_{p}}B \right),  
\end{aligned}
\end{equation}
where $p\in \left[ 0,1 \right]$, $r=\min \left\{ p ,1-p  \right\}$ and $R=\max \left\{ p,1-p \right\}$.
\end{corollary}
\begin{proof}
If we apply positive linear map $\Phi $ in \eqref{23} we infer
\begin{equation}\label{5}
{{K}^{r}}\left( h',2 \right)\Phi \left( A{{\#}_{p}}B \right)\le \Phi \left( A{{\nabla }_{p}}B \right)\le {{K}^{R}}\left( h,2 \right)\Phi \left( A{{\#}_{p}}B \right).
\end{equation}
On the other hand, if we take $A=\Phi \left( A \right)$ and $B=\Phi \left( B \right)$ in \eqref{23} we can write 
\begin{equation}\label{6}
{{K}^{r}}\left( h',2 \right)\Phi \left( A \right){{\#}_{p}}\Phi \left( B \right)\le \Phi \left( A{{\nabla }_{p}}B \right)\le {{K}^{R}}\left( h,2 \right)\Phi \left( A \right){{\#}_{p}}\Phi \left( B \right).
\end{equation}
Now, combining inequality \eqref{5}  and inequality \eqref{6}, we deduce the desired inequalities \eqref{14}.
\end{proof}
\begin{remark}
It is well-known that the generalized Kantorovich constant $K\left( h,p \right)$ {{\cite[Definition 1]{b}}} is defined by
\begin{equation}\label{ka}
K\left( h,p \right):=\frac{{{h}^{p}}-h}{\left( p-1 \right)\left( h-1 \right)}{{\left( \frac{p-1}{p}\frac{{{h}^{p}}-1}{{{h}^{p}}-h} \right)}^{p}},
\end{equation}
for all $p\in \mathbb{R}$. By virtue of a generalized Kantorovich constant, in the matrix setting, Bourin {\it et al.} in {{\cite[Theorem 6]{referee2}}} gave the following reverse of Ando's inequality for a positive linear map:

Let $A$ and $B$ be positive operators such that $mA\le B\le MA$, and let $\Phi $ be a positive linear map. Then 
\begin{equation}\label{c1}
\Phi \left( A \right){{\#}_{p}}\Phi \left( B \right)\le \frac{1}{K\left( {{h}},p \right)}\Phi \left( A{{\#}_{p}}B \right),\quad p\in \left[ 0,1 \right]
\end{equation}
where $h=\frac{M}{m}$. The above result naturally extends  one proved in Lee {{\cite[Theorem 4]{referee1}}} for $p=\frac{1}{2}$.

Of course the constant $\frac{{{K}^{R}}\left( h,2 \right)}{{{K}^{r}}\left( h',2 \right)}$ is not better than $\frac{1}{K\left( \frac{h}{h'},p \right)}$. Concerning the sharpness of the estimate \eqref{c1}, see {{\cite[Lemma 7]{referee2}}}. 
However our bounds on $\Phi(A)\#_p\Phi(B)$ are calculated by the original Kantrovich constant $K(h,2)$ without the generalized one $K(h,p)$. It is also interesting our bounds on $\Phi(A)\#_p\Phi(B)$ are expressed by $\Phi(A\nabla_p B)$  with only one constant either $h$ or $h'$. 
\end{remark}
\vskip 0.3 true cm
After discussion on inequalities related to the operator mean with positive linear map, we give a result on Tsallis relative operator entropy (which is the main theme in this paper) with a positive linear map.
It is well-known that Tsallis relative operator entropy has the following information monotonicity:
\begin{equation}\label{51}
\Phi \left( {{T}_{p}}\left( A|B \right) \right)\le {{T}_{p}}\left( \Phi \left( A \right)|\Phi \left( B \right) \right).
\end{equation}
Utilizing \eqref{25}, we have the following counterpart of \eqref{51}:
\begin{theorem}\label{e}
Let $A,B$ be two invertible positive operators. Let $\Phi $ be normalized positive linear map on $\mathcal{B}\left( \mathscr{H} \right)$, then
\begin{equation}\label{26}
\begin{aligned}
  & \frac{2r}{p}\left( \Phi \left( A\nabla B \right)-\Phi \left( A \right)\#\Phi \left( B \right) \right)+{{T}_{p}}\left( \Phi \left( A \right)|\Phi \left( B \right) \right) \\ 
 & \le \Phi \left( B-A \right) \\ 
 & \le \frac{2R}{p}\left( \Phi \left( A\nabla B \right)-\Phi \left( A\#B \right) \right)+\Phi \left( {{T}_{p}}\left( A|B \right) \right), \\ 
\end{aligned}
\end{equation}
where $p\in \left( 0,1 \right]$, $r=\min \left\{ p ,1-p  \right\}$ and $R=\max \left\{ p,1-p \right\}$.
\end{theorem}
\begin{proof}
Using the method of the proof of Proposition \ref{prop} and Corollary \ref{7} for the inequality \eqref{25}, we can obtain that 
\[\begin{aligned}
  & 2r\left( \Phi \left( A\nabla B \right)-\Phi \left( A \right)\#\Phi \left( B \right) \right)+\Phi \left( A \right){{\#}_{p}}\Phi \left( B \right) \\ 
 & \le \Phi \left( A{{\nabla }_{p}}B \right) \\ 
 & \le 2R\left( \Phi \left( A\nabla B \right)-\Phi \left( A\#B \right) \right)+\Phi \left( A{{\#}_{p}}B \right). \\ 
\end{aligned}\]
A simple calculation shows that
\[\begin{aligned}
  & \frac{2r}{p}\left( \Phi \left( A\nabla B \right)-\Phi \left( A \right)\#\Phi \left( B \right) \right)+\frac{\Phi \left( A \right){{\#}_{p}}\Phi \left( B \right)-\Phi \left( A \right)}{p} \\ 
 & \le \Phi \left( B-A \right) \\ 
 & \le \frac{2R}{p}\left( \Phi \left( A\nabla B \right)-\Phi \left( A\#B \right) \right)+\frac{\Phi \left( A{{\#}_{p}}B \right)-\Phi \left( A \right)}{p}. \\ 
\end{aligned}\]
Apparently, the above inequality is equivalent to inequality \eqref{26}. The proof is completed.
\end{proof}
\vskip 0.3 true cm
The following example may render the above statement clearer.
\begin{example}
Let $\Phi \left( X \right)={{U}^{*}}XU\left( X\in {{\mathscr{M}}_{2}} \right)$ where 	$U=\left( \begin{matrix}
   \frac{\sqrt{2}}{2} & \frac{\sqrt{2}}{2}  \\
   -\frac{\sqrt{2}}{2} & \frac{\sqrt{2}}{2}  \\
\end{matrix} \right)$. If $A=\left( \begin{matrix}
   2 & -1  \\
   -1 & 1  \\
\end{matrix} \right)$ and $B=\left( \begin{matrix}
   6 & 2  \\
   2 & 4  \\
\end{matrix} \right)$ and $p=\frac{1}{4}$ we get
\[\frac{2r}{p}\left( \Phi \left( A\nabla B \right)-\Phi \left( A \right)\#\Phi \left( A \right) \right)+\left( {{T}_{p}}\left( \Phi \left( A \right)|\Phi \left( A \right) \right) \right)= \left( \begin{matrix}
   0.486 & 0.443  \\
   0.443 & 5.638  \\
\end{matrix} \right)\]

\[\Phi \left( B-A \right)=\left( \begin{matrix}
   0.5 & 0.5  \\
   0.5 & 6.5  \\
\end{matrix} \right)\]
and
\[\frac{2R}{p}\left( \Phi \left( A\nabla B \right)-\Phi \left( A\#B \right) \right)+\Phi \left( {{T}_{p}}\left( A|B \right) \right)= \left( \begin{matrix}
   0.562 & 0.951  \\
   0.951 & 13.521  \\
\end{matrix} \right)\]
which shows that  inequality \eqref{26} is true.
\end{example}
\vskip 0.3 true cm

Tsallis relative entropy ${{D}_{p}}\left( A||B \right)$ for two positive operators $A$ and $B$ is defined by:
\[{{D}_{p}}\left( A||B \right):=\frac{Tr\left[ A \right]-Tr\left[ {{A}^{1-p}}{{B}^{p}} \right]}{p},\quad p\in \left( 0,1 \right].\]
In information theory, relative entropy (divergence) is usually defined for density operators which are positive operators with unit trace. However we consider Tsallis relative entropy defined for positive operators to derive the relation with Tsallis relative operator entropy.
If $A$ and $B$ are positive operators, then
\begin{equation}\label{54}
Tr\left[ A-B \right]\le {{D}_{p}}\left( A||B \right)\le -Tr\left[ {{T}_{p}}\left( A|B \right) \right],\quad p\in \left( 0,1 \right].
\end{equation}
Note that the first inequality of \eqref{54} is due to Furuta {{\cite[Proposition F]{9}}} and the second inequality is due to Furuichi {\it et al.} {{\cite[Theorem 2.2]{22}}}. 

As a direct consequence of Theorem \ref{e}, we have the following interesting relation.
\begin{corollary}
Let $A,B$ be two positive operators on a finite dimensional Hilbert space $\mathscr{H}$, then
\begin{equation}\label{56}
\begin{aligned}
  & \frac{2R}{p}\left( Tr\left[ A\#B \right]-\frac{Tr\left[ A+B \right]}{2} \right)-Tr\left[ {{T}_{p}}\left( A|B \right) \right] \\ 
 & \le Tr\left[ A-B \right] \\ 
 & \le \frac{2r}{p}\left( \sqrt{Tr\left[ A \right]Tr\left[ B \right]}-\frac{Tr\left[ A+B \right]}{2} \right)+{{D}_{p}}\left( A||B \right), \\ 
\end{aligned}
\end{equation}
where $p\in \left( 0,1 \right]$, $r=\min \left\{ p ,1-p  \right\}$ and $R=\max \left\{ p,1-p \right\}$.
\end{corollary}
\begin{proof}
Taking $\Phi \left( X \right)=\frac{1}{\dim \mathscr{H}}Tr\left[ X \right]$ in \eqref{14}, we have
\begin{equation}\label{20}
\begin{aligned}
  & \frac{2r}{p}\left( \frac{Tr\left[ A+B \right]}{2}-\sqrt{Tr\left[ A \right]Tr\left[ B \right]} \right)+\frac{{{\left( Tr\left[ A \right] \right)}^{1-p}}{{\left( Tr\left[ B \right] \right)}^{p}}-Tr\left[ A \right]}{p} \\ 
 & \le Tr\left[ B-A \right] \\ 
 & \le \frac{2R}{p}\left( \frac{Tr\left[ A+B \right]}{2}-Tr\left[ A\#B \right] \right)+Tr\left[ {{T}_{p}}\left( A|B \right) \right]. \\ 
\end{aligned}
\end{equation}
It is not too difficult to see that \eqref{20} can be also reformulated in the following way
\[\begin{aligned}
  & \frac{2R}{p}\left( Tr\left[ A\#B \right]-\frac{Tr\left[ A+B \right]}{2} \right)-Tr\left[ {{T}_{p}}\left( A|B \right) \right] \\ 
 & \le Tr\left[ A-B \right] \\ 
 & \le \frac{2r}{p}\left( \sqrt{Tr\left[ A \right]Tr\left[ B \right]}-\frac{Tr\left[ A+B \right]}{2} \right)+\frac{Tr\left[ A \right]-{{\left( Tr\left[ A \right] \right)}^{1-p}}{{\left( Tr\left[ B \right] \right)}^{p}}}{p}. \\ 
\end{aligned}\]
Now, having in mind that $Tr\left[ {{A}^{1-p}}{{B}^{p}} \right]\le {{\left( Tr\left[ A \right] \right)}^{1-p}}{{\left( Tr\left[ B \right] \right)}^{p}}$, we infer
\[\begin{aligned}
  & \frac{2R}{p}\left( Tr\left[ A\#B \right]-\frac{Tr\left[ A+B \right]}{2} \right)-Tr\left[ {{T}_{p}}\left( A|B \right) \right] \\ 
 & \le Tr\left[ A-B \right] \\ 
 & \le \frac{2r}{p}\left( \sqrt{Tr\left[ A \right]Tr\left[ B \right]}-\frac{Tr\left[ A+B \right]}{2} \right)+\frac{Tr\left[ A \right]-Tr\left[ {{A}^{1-p}}{{B}^{p}} \right]}{p}, \\ 
\end{aligned}\]
which, in turn, leads to \eqref{56}.
\end{proof}

\begin{remark}
If $A$ and $B$ are density operators, then the second inequality of (\ref{56}) implies the non-negativity of Tsallis relative entropy, $D_p(A||B) \geq 0$.
If $p \in [\frac{1}{2},1]$, then the first inequality of (\ref{56}) implies
$$
Tr[T_{1/2}(A|B)] \leq Tr[T_{p}(A|B)].
$$
If $p \in (0,\frac{1}{2}]$, then  the first inequality of (\ref{56}) also implies
$$
(1-p) Tr[T_{1/2}(A|B)] +(2p-1) Tr[B-A] \leq p  Tr[T_{p}(A|B)].
$$
\end{remark}



\end{document}